\documentclass[review]{elsarticle}

\usepackage{lineno,hyperref}
\modulolinenumbers[5]

\journal{Journal of \LaTeX\ Templates}

\usepackage{natbib}
\usepackage{rotating}
\usepackage{graphics,pst-3d,pst-node,pst-plot,pst-text}
\usepackage[normalem]{ulem}
 \bibpunct[, ]{[}{]}{,}{a}{}{,}%
 \def\bibfont{\small}%
 \def\newblock{\ }%
\newcommand{\CC}{{\cal C}}

\newcommand{\RR}{{\cal R}}

\newcommand{\UU}{{\cal U}}
\def\R{I\!\!R}

\usepackage{amsmath}
\usepackage{amsthm}
 \newtheorem{proposition}{Proposition}[section]
 \newtheorem{lemma}{Lemma}[section]
 \newtheorem{corollary}{Corollary}[section]
 \newtheorem{definition}{Definition}[section]

\begin{document}

\begin{frontmatter}


\title{Robust DEA efficiency scores: A probabilistic/combinatorial approach\footnote{This research has been partly supported by the Spanish Ministerio de Economía y Competitividad, through grants MTM2013-43903-P and  MTM2015-68097-P(MINECO/FEDER), and Ministerio de Economía, Industria y Competitividad, through grant MTM2016-76530-R (AEI/FEDER, UE).}}

\author{Mercedes Landete, Juan F. Monge\footnote{Corresponding author} and Jos\'e L. Ruiz}
\address{Centro de Investigaci\'{o}n Operativa, Universidad Miguel Hern\'{a}ndez. \\ Avda. de la Universidad, s/n, 03202-Elche (Alicante), Spain.\\
 E-mail addresses: \,  landete@umh.es \, (M. Landete),\, monge@umh.es\, (J.F. Monge),\, jlruiz@umh.es\, (J.L. Ruiz)}

%
%
%
\begin{abstract}
In this paper we propose robust efficiency scores for the scenario
in which the specification of the inputs/outputs to be included in
the DEA model is modelled with a probability distribution. This
probabilistic approach allows us  to obtain three different  robust
efficiency scores: the Conditional Expected Score, the
Unconditional Expected Score and the  Expected score under the
assumption of Maximum Entropy principle.  The calculation of the three efficiency scores involves
the resolution of  an exponential number of linear problems. The
algorithm presented in this paper allows to solve over 200
millions of linear problems in an affordable time when considering
up 20 inputs/outputs and 200 DMUs. The approach proposed is illustrated with an application to the assessment of professional tennis players.  \end{abstract}


\end{frontmatter}


\section{Introduction}

Data Envelopment Analysis (DEA), as introduced in
\cite{Charnes1978}, assesses relative efficiency of decision making
units (DMUs) involved in production processes. DEA models provide
efficiency scores of the DMUs in the form of a weighted sum of
outputs to a weighted sum of inputs. These scores are the result of
an evaluation of each unit within a technology which is empirically
constructed from the observations by assuming some postulates such
as convexity, constant or variable returns to scale and free
disposability. The selection of the inputs and outputs to be considered in the analysis provides a description of the underlying technology, thus becoming one of the key issues of model specification in DEA.  In practical applications, the
prior knowledge and experience may lead the analyst to select some
variables considered as essential to represent this technology.
However, as discussed in \cite{pastor2002}, there are often other
variables whose inclusion in the model the analyst is not always
sure about. This situation can be addressed in different ways. The
methods for the selection of variables constitute an important body
of research to deal with this issue. The idea is to complement the
prior knowledge and experience with information provided by the
data, so that these methods may help make a decision about the
candidate variables. In this line, the F-tests developed by
\citet{banker1993,banker1996-2} and those in \cite{pastor2002} allow to statistically evaluating the role of inputs/outputs. These tests are empirically
analyzed (and compared among themselves) with simulations in \citet{sirvent2005}. 
See also \cite{natajara2011}, which provides comparisons between some of the methods widely used to guide variable selection: the tests in \cite{pastor2002}, Principal Component Analysis (PCA-DEA), regression-based tests and bootstrapping. To the same end, \cite{wagner2007} propose a stepwise approach. In \cite{li2016} it is proposed a method based on the Akaike''s information criteria (AIC), which mainly focuses on assessing the importance of subset of original variables rather than testing the marginal role of variables one by one as in many other methods.

Correlation either between efficiency scores and variables (for
their incorporation into the model) or between variables (in order
to remove redundant factors) has also been used for the selection of
variables, although it has been widely deemed as a criterion of
limited value. See \cite{jenkins2003} for discussions. This latter paper proposes instead an approach based on partial covariance. \cite{eskelinen2017} compares the approach in \cite{jenkins2003} with that in \cite{pastor2002} in an empirical retail bank context. Unlike previous research, \cite{edirisinghe2010} develop a method for the selection of variables that employs a reward variable observed exogenous to the operation of DMUs. See also \cite{luo2012}, which uses the concept of cash value added (CVA) for choosing variables.

A different approach, which is the one we follow in
the present paper, is based on efficiency scores which are
robust against the selection of the variables, while at the same
time taking into account the inherent uncertainty on the inclusion
of some inputs/outputs in the models. In the literature, some
authors have undertaken some exploratory work to examine the
robustness of results as they relate to the production function
specification. \cite{roll1989} present one such study in which ten
different combinations of outputs are tried out with three inputs to
evaluate the efficiency of maintenance units in the Israeli Air
Force; \cite{Valdmanis1992} tests ten different specifications of a
hospital efficiency model, and \cite{ahn1993} test four different
variable sets for each of the DEA models considered in their study
on higher education sector efficiency. For his part,
\cite{smith1997} investigates the implications of model
misspecification by means of a simulation study. See also
\cite{Galagedera2003} for another simulation study with large sample
which, in particular, analyze the effects of omission of relevant
input and inclusion of irrelevant input variables on technical
efficiency estimates.

    The approach we propose goes one step further than the exploratory work done in the papers just mentioned, in the sense that we not only examine the efficiency scores that result from several combinations of inputs and outputs but we take into account all the scenarios associated with all of the specifications of inputs and outputs that could be considered once a given set of candidate variables is determined. In addition, we allow for the uncertainty regarding the inclusion of variables in the DEA models. The idea is the following: The inclusion of an input/output in the set of selected variables is modelized here through the probability of that variable being considered in the DEA model. For example, if such probability is 0.8, this could be interpreted as saying that 80\% of experts would include the corresponding variable in the DEA model. As a result, each specification of the inputs and outputs to be included in the DEA model has a probability of occurrence and, therefore, the efficiency score of a DMU would be a random variable, which takes as values the DEA efficiency scores associated with each specification of inputs and outputs with some probability. The robust efficiency score of a given DMU is then defined as the expected value of that random variable.

The  consideration of all combinations of inputs/outputs gives rise to 
an exponential number of problems that must be solved.  To solve
such  large number of problems, an efficient algorithm is needed.
In this paper  an exact algorithm is  developed, which allows us  to
solve over 200 millions of linear problems when considering up 20
inputs/outputs and 200 DMU's. This  algorithm reduces   the time
and the number of problems to solve in half, approximately.

    We illustrate the use of the proposed approach with an application to the assessment of professional tennis players. The Association of Tennis Professionals (ATP) assesses players through the points earned in the different tournaments they play during the season. Therefore, ATP assesses  the competitive performance of players. However,  ATP also provides statistics regarding their game performance. For instance, its official webpage reports data regarding 9 game factors such as the percentage of 1st serve points won or the percentage of return games won. Obviously, it would be interesting to have available an index of the game overall performance of players that aggregates into a single scalar the information provided by the statistics of the factors that are considered. The DEA approach we propose provides a score of the player game performance which is robust against the selection of game factors that is considered for the analysis. \cite{ruiz2013} also deal with the assessment of game performance of tennis players, but with an approach based on the cross-efficiency evaluation that consider all of the 9 game factors available in the ATP statistics.

The paper is organized as follows: In Section 2  a short introduction of DEA, through the
original CCR  (Charnes, Cooper and Rhodes) model, is presented. 
In Section 3 we define the the robust efficiency score,  and  Section 4  presents three  robust DEA efficiency scores for a
probabilistic specification of the inputs and outputs.  The exact solution algorithm used for the calculation of
the robust scores is described in Section 5. In Section 6 the
proposed algorithm is used for obtaining the robust scores in a
case study. Finally, some conclusions and outlines for future work
are given in Section 7.

\section{DEA Efficiency Scores}

In a DEA efficiency analysis, we have $n$ $DMUs$ which use $m$ inputs to produce $s$ outputs. Each $DMU_j$ can be described by means of the vector
$(X_j,Y_j)=(x_{1j},\dots,x_{mj},y_{1j},\dots,y_{sj})$, $j=1,\dots,n.$

As said before, the DEA models assess efficiency with reference to
an empirical technology or production possibility set  which is
constructed from the observations by assuming some postulates. For
instance, if we assume convexity, constant returns to scale, and
free disposability (which means that if we can produce $Y$ with
$X$, then we can both produce less than $Y$ with $X$ and $Y$ with
more than $X$), then it can be shown that the technology can be characterized as  the
set $T=\{(X,Y)\in \R_+^{m+s}  / X\geq \sum_{j=1}^{n}\lambda_j X_j,
\, Y\leq \sum_{j=1}^{n}\lambda_j Y_j,  \, \lambda_j \geq 0, \,
j=1,\dots, n \} .$ The original DEA model by \cite{Charnes1978},
the CCR model, provides as measure of the relative efficiency of a
given $DMU_0$ the minimum value $\theta_0$ such that $(\theta_0
X_0, Y_0)\in T$. Therefore, this value can obviously be obtained
by solving the following linear programming  problem.

\begin{equation}
\begin{aligned}
& \min & & \theta_0 \\
& \text{s.t.}& & \sum_{j=1}^{n}\lambda_j x_{ij}  \leq   \theta_0 x_{i0}, \qquad  i=1,\dots,m \\
&& & \sum_{j=1}^{n}\lambda_j y_{rj}  \geq  y_{r0}, \qquad    r=1,\dots,s \\
&&& \lambda_j \geq  0, \qquad    \forall j
\end{aligned}
\label{CRS}
\end{equation}

\noindent which is the so-called primal envelopment formulation of
the CCR model. Thus, $DMU_0$ is said to be efficient if, and only
if, its efficiency score equals 1. Otherwise, it is inefficient,
and the lower the efficiency score, the lesser its efficiency. The
model in \cite{Charnes1984}, the so-called BCC model,  is that resulting from eliminating the
constant returns to scale postulate and allowing for variable
returns to scale  in the production possibility set. Its
formulation is the linear problem resulting from adding the
constraint $\sum_{j=1}^n \lambda_j =1$ to (\ref{CRS}).

\section{Robust DEA Efficiency Scores:  A probabilistic approach}

Throughout the paper we suppose that we have a set of candidate
variables to be included in the efficiency model (\ref{CRS}),
$\CC=\{z_1,\dots,z_q\}$, which can be either inputs or outputs. It
is assumed that the probability of including in (\ref{CRS}) a
given candidate variable, say $z_c$, is $p_c$. Thus, the inclusion
of each of the variables in $\CC$ into (\ref{CRS}) can be
determined by means of some independent random variables $B_c$
distributed Bernoulli, $Be(p_c)$, $c=1,\dots,q$. As a result, all
the scenarios associated with all the possible specifications of
(\ref{CRS}) are determined by the random vector
$B=(B_1,\dots,B_q)$. If we denote by $p=(p_1,\dots,p_q)$, then the
probability distribution of $B$ is

\[ \displaystyle P_p(B=b)= \prod_{c=1}^{q}  p_c^{b_c} (1-p_c)^{(1-b_c)}\hspace{2cm}b=(b_1,\dots,b_q)\in \{0,1\}^q.\]

Let  $\theta_0^b$ be the efficiency score of $DMU_0$ provided by
(\ref{CRS}) when the specification of the model is determined by
$b\in \{0,1\}^q.$ We denote by $\Theta_0$ the random variable
which takes the value $\theta_0^b.$ Then, the expected efficiency
score of $DMU_0$ is

\begin{align}
\displaystyle E_p(\Theta_0)= \sum_{b\in \{0,1\}^q}    \left (\prod_{c=1}^{q}  p_c^{b_c} (1-p_c)^{(1-b_c)}  \right )  \,      \theta_0^b
\label{esperanza1.1}
\end{align}

For consistency, we define $\theta_0^{b=(0,\dots,0)} =1$, i.e, we
assume that  all DMUs are efficient when the input/output set is
empty. The value of $\theta_0^{b=(0,\dots,0)}$  is not relevant to
compare the expected efficiency scores between DMUs, because
$\theta_0^{b=(0,\dots,0)}$ is the same constant for all of them.

These expected values can be seen as DEA efficiency scores which are robust against the selection of variables that is made for the efficiency model.

\section{The specification of $p$}

The key for obtaining the expected efficiency scores
(\ref{esperanza1.1}) is in the specification of $p$. Three
different approaches to deal with this issue are proposed below.

\subsection{Using expert opinion}

The probability of selection of candidate variables can be
determined by using information from experts, if available. The
values $p_c's$ can be set reflecting the personal belief of a
given expert regarding the importance to be attached to  the corresponding
variables $z_c's$ in the underlying production process.
Alternatively, these probabilities can be estimated. If several
experts are asked to give their opinion about whether or not to
include a given $z_c$ in (\ref{CRS}) (in presence of the remaining
variables), then the proportion of those in favor of such
inclusion provides an estimation of $p_c$.

\subsection{Maximizing the Entropy}

The definition of $entropy$ was introduced by \cite{shannon1948}. In the probabilistic context, the entropy $H(p)$ is a measure of the information provided by $p$, where high values of entropy corresponds to less information o more uncertainty, i.e., the maximum principle Entropy is used in Statistics to obtain a value for the parameters with the least informative distribution assumptions.  If information from experts is not available and the probabilities of selection of candidate variables are unknown, we can obtain the value of $p$ that maximizes the entropy in the context of the approach previously set out for providing robust DEA efficiency scores.

The Entropy function associated with the discrete random variable  $\Theta_0$ is:

\[ \displaystyle H(p) = - \sum_{b\in \{0,1\}^{q}}  P_p(\Theta_0=\theta_0^{b})  \, \log (P_p(\Theta_0=\theta_0^{b})  )  \]

where $P_p(\Theta_0=\theta_0^{b}) =  \prod_{c=1}^{q}  p_c^{b_c} (1-p_c)^{(1-b_c)}. $


\begin{lemma}{\bf  [\cite{entropy} and \cite{conrad}]}
Suppose that a random variable $X$ takes exactly $\ell$
values with positive probability. Then $H(X)\leq \log \ell$.
\end{lemma}

\begin{proposition}  The  entropy function associated with $\Theta_0$ has a maximum in
the probability vector $p^*=(1/2,\dots,1/2).$ That is,  $H(p)\le H(p^*)$ for
all $p.$

\end{proposition}

\begin{proof} Applying Lemma 1, it is sufficient to prove that
 $H(p^*)=\log 2^q$, because  $ 2^{q}$ is the number of possible
realizations of variable $\Theta_0$ (that is, the number of
scenarios determined by all the possible selections of
inputs/outputs).

Since $ P_{p^*}(\Theta_0=\theta_0^{b})  = \displaystyle  \prod_{c=1}^{b} \left (\frac{1}{2}
\right )^{b_c} \left(\frac{1}{2} \right
)^{(1-b_c)}=\frac{1}{2^b}$, then \\ $
 \displaystyle H(p^*) = -\sum_{b\in \{0,1\}^{q}} \frac{1}{2^{q}} \log (\frac{1}{2^{q}}) =2^{q} \frac{1}{2^{q}}\log (2^{q}) =  \log
(2^{q}).$    
\end{proof}

\begin{corollary} The maximum entropy expected efficiency score is
\begin{align}
\displaystyle E^e(\Theta_0) =   \frac{1}{2^{q}}  \sum_{b\in \{0,1\}^{q}}  \theta_0^b
\label{esperanza2}
\end{align}
\end{corollary}

\begin{proof}
The score (\ref{esperanza2}) is simply the result of calculating (\ref{esperanza1.1}) with $p=(1/2,\dots,1/2)$.
\end{proof}

This corollary shows that the average of efficiency scores across all the scenarios resulting from all the specifications of model (\ref{CRS}) is the one associated with the specification of $p$ that maximizes the entropy.

\subsection{A Bayesian approach}

In this subsection we develop a Bayesian approach as an alternative for the specification of $p$ when  it is unknown. This means that the probabilities of selection of candidate variables are assumed to be random variables in $[0,1]$. Denote by $\mathrm{P}=(\mathrm{P}_1,\dots,\mathrm{P}_q)$ the random vector consisting of the independent random variables associated with the probability of selection of each of the candidates.
Let $f$ be the joint probability density function of $\mathrm{P}$, which can be expressed as $f(\mathrm{p})=\prod_{c=1}^q f_c(\mathrm{p}_c), \mathrm{p}=(\mathrm{p}_1,\dots,\mathrm{p}_q)\in [0,1]^q$, $f_c$ being the probability density function of $\mathrm{P}_c$, $c=1,\dots,q$.

We need to introduce the following two elements for the subsequent developments

\begin{definition} The unconditional probability function of $\Theta_0$ is defined as
\begin{align}
\displaystyle P(\Theta_0=\theta_0^b) =  \int_{\mathrm{p} \in [0,1]^q}  P(\Theta_0=\theta_0^b\, | \, \mathrm{P}=\mathrm{p}) \, f(\mathrm{p}
) \,  \text{d}\mathrm{p}
\label{uncondprobablity}
\end{align}
\end{definition}

Alternatively, the unconditional probability function of
$\Theta_0$ can be reexpressed as

\begin{align}
\displaystyle P(\Theta_0=\theta_0^b) =  E\left [
P(\Theta_0=\theta_0^{b}\, | \, \mathrm{P}) \right ]
\label{uncondprobablity2}
\end{align}

That is, $P(\Theta_0=\theta_0^b)$ can be seen as the expected conditional probability to $\mathrm{p}$.

\begin{definition} The unconditional expected efficiency score of $\Theta_0$ is defined as
\begin{align}
\displaystyle E(\Theta_0) =  \sum_{b\in \{0,1\}^{q}}  \theta_0^{b} \, P(\Theta_0=\theta_0^{b})
\label{uncondeffscorecomoespincond}
\end{align}
\end{definition}

The following two propositions hold

\begin{proposition}
$E(\Theta_0) =E\left [ E(\Theta_0\, | \,\mathrm{P}) \right ]  $
\label{uncondeffscorecomoespespcond}
\end{proposition}
\begin{proof}
It follows directly from (\ref{uncondeffscorecomoespincond}).
\end{proof}

\begin{proposition}
$E(\Theta_0)=
 \sum_{b\in \{0,1\}^{q}}   \ E\left [ P(\Theta_0=\theta_0^{b}\, | \, \mathrm{P}) \right ] \, \theta_0^{b} $
\label{uncondeffscorecomomediapondscores}
\end{proposition}

\begin{proof}
It follows directly from (\ref{uncondprobablity2}).
\end{proof}

In a Bayesian approach, the probabilities are often assumed to follow a $beta$ distribution. In that case, the unconditional expected efficiency score is the following

\begin{proposition}
If $\mathrm{P}_c$ follows a $beta$ distibution, $\beta(\alpha_c,\gamma_c)$, $c=1,\dots,q$, then
\begin{align}
\displaystyle E_{\beta}(\Theta_0) =  \sum_{b\in \{0,1\}^{q}}    \frac{ \displaystyle  \prod_{c=1}^{q}  \alpha_c^{b_c}   \prod_{c=1}^{q}  \gamma_c^{1-b_c} }{ \displaystyle \prod_{c=1}^{q} (\alpha_c+\gamma_c)} \,   \theta_0^{b}
\label{esperanza3.1}
\end{align}

\end{proposition}

\begin{proof}
 \begin{align}
 \displaystyle P(\Theta_0=\theta_0^b) = & \int_{\mathrm{p} \in [0,1]^q}  P(\Theta_0=\theta_0^b\, | \, \mathrm{P}=\mathrm{p})  \, f(\mathrm{p}) \,\text{d}\mathrm{p} = \nonumber\\
 &=  \int_{0}^{1} \hspace{-8pt} \cdots \hspace{-5pt} \int_{0}^{1}  \prod_{c=1}^{q}  \mathrm{p}_t^{b_c}  \, (1-\mathrm{p}_t)^{(1-b_c)}  \,   \frac{\Gamma(\alpha_c+\gamma_c)}{\Gamma(\alpha_c) \Gamma(\gamma_c)}  \, \mathrm{p}_t^{\alpha_c-1}  \, (1-\mathrm{p}_t)^{\gamma_c-1} \,\text{d}\mathrm{p}_1 \cdots   \text{d}\mathrm{p}_{q} = \nonumber\\
 &=  \int_{0}^{1} \hspace{-8pt} \cdots \hspace{-5pt} \int_{0}^{1}  \prod_{c=1}^{q}  \mathrm{p}_t^{\alpha_c+b_c-1} (1-\mathrm{p}_t)^{(\gamma_c-b_c)}   \frac{\Gamma(\alpha_c+\gamma_c)}{\Gamma(\alpha_c) \Gamma(\gamma_c)}  \,\text{d}\mathrm{p}_1 \cdots   \text{d}\mathrm{p}_{q} = \nonumber\\
 &=  \int_{0}^{1} \hspace{-8pt} \cdots \int_{0}^{1} \, \prod_{c=2}^{q}  \mathrm{p}_t^{\alpha_c+b_c-1}  \, (1-\mathrm{p}_t)^{(\gamma_c-b_c)} \, \frac{\Gamma(\alpha_c+\gamma_c)}{\Gamma(\alpha_c) \Gamma(\gamma_c)}  \cdot \nonumber \\
 &  \qquad \qquad \qquad  \cdot  \left(  \int_{0}^{1}  \,   \mathrm{p}_1^{\alpha_1+b_1-1}  \, (1-\mathrm{p}_1)^{(\gamma_1-b_1)}  \,   \frac{\Gamma(a_1+b_1)}{\Gamma(a_1) \Gamma(b_1)}  \,\text{d}p_1 \right )  \text{d}\mathrm{p}_2  \cdots   \text{d}\mathrm{p}_q = \nonumber\\
  &=\displaystyle   \frac{\alpha_1^{b_1}  \, \gamma_1^{1-b_1} }{\alpha_1+\gamma_1}\int_{0}^{1} \hspace{-8pt} \cdots \int_{0}^{1} \, \prod_{c=2}^{q}  \mathrm{p}_t^{\alpha_c+b_c-1}  \, (1-\mathrm{p}_t)^{(\gamma_c-b_c)} \, \frac{\Gamma(\alpha_c+\gamma_c)}{\Gamma(\alpha_c) \Gamma(\gamma_c)}    \,  \text{d}\mathrm{p}_2  \cdots   \text{d}\mathrm{p}_{q} = \nonumber\\
  & =   \frac{ \displaystyle  \prod_{c=1}^{q}  \alpha_c^{b_c}  \,   \prod_{c=1}^{q}  \gamma_c^{1-b_c} }{ \displaystyle \prod_{c=1}^{q} (\alpha_c+\gamma_c)} \nonumber \\
& \nonumber 
 \end{align}
\begin{align}
E_{\beta}(\Theta_0) &=\sum_{b\in \{0,1\}^{q}}   E\left [ P(\Theta_0=\theta_0^{b}\, | \, \mathrm{P}) \right ]  \,   \theta_0^{b}  =   \sum_{b\in \{0,1\}^{q}}    \frac{ \displaystyle  \prod_{c=1}^{q}  \alpha_c^{b_c}   \prod_{c=1}^{q}  \gamma_c^{1-b_c} }{ \displaystyle \prod_{c=1}^{q} (\alpha_c+\gamma_c)} \,   \theta_0^{b} \nonumber
\end{align}

\end{proof}

The uniform distribution is a particular case of the beta.
Specifically, $U[0,1]=\beta(1,1)$. 

\begin{corollary}
If $\mathrm{P}$ is a random vector consisting of independent random variables distributed $\mathrm{U}[0,1]$, then
  \begin{align}
E_{\UU}(\Theta_0) = \frac{1}{2^{q}}  \sum_{b\in \{0,1\}^{q}} \theta_0^b \label{esperanza3.2}
\end{align}
\end{corollary}

Note that in this case the \emph{unconditional expected efficiency score} is again the average of efficiency scores across all the scenarios resulting from all the specifications of model (\ref{CRS}), like the \emph {maximum entropy expected efficiency score}.

It can also be considered the case in which we do not distinguish
between the probabilities of selection associated with all the
candidate variables. If we assume
$\mathrm{P}_1=\dots=\mathrm{P}_q=\overline{\mathrm{P}}$ to follow
a distribution $U[0,1]$, and we denote by
$E_{\overline\UU}(\Theta_0^{})$ the unconditional expected
efficiency score in that case, then
(\ref{uncondeffscorecomoespespcond}) and
(\ref{uncondeffscorecomomediapondscores}) become in the unconditional expected efficiency score presented  in the next results:

\begin{corollary}If $\mathrm{P}_1=\dots=\mathrm{P}_q=\overline{\mathrm{P}}$ follows an uniform distribution $U[0,1]$, then the unconditional expected score of $\Theta_0$, $E_{\overline\UU}(\Theta_0^{})$, is the expected of the conditional expected score $E(\Theta_0^{}|\mathrm{P})$.
\begin{align}
\displaystyle  E_{\overline\UU}(\Theta_0^{})= \int_{0}^{1} E_{\overline{\mathrm{p}}}(\Theta_0 )  \, \text{d}\overline{\mathrm{p}}   \label{col4}
  \end{align}
\end{corollary}
\begin{proof}{See Proposition \ref{uncondeffscorecomoespespcond}.}
\end{proof}

\begin{corollary} The unconditional expected score of $\Theta_0
$, $E_{\overline\UU}(\Theta_0^{})$, is given by :
\begin{align}
\displaystyle  E_{\overline\UU}(\Theta_0^{}) = \frac{1}{q+1} \sum_{b\in \{0,1\}^{q}}  \frac{1}{ \displaystyle \binom{q}{\sum_c b_c }} \,\theta_0^b =
 \frac{1}{q+1} \sum_{i=0}^{q} \left(\sum_{\substack{b\in \{0,1\}^{q} \\ \sum_c b_c= i}}  \frac{1}{ \displaystyle \binom{q}{i }} \,\theta_0^b \right )
 \label{esperanza3.4}
  \end{align}
\end{corollary}
\begin{proof}{See Proposition \ref{uncondeffscorecomomediapondscores}.}
\end{proof}

Note that in this case the weight  attached to  each score
$\theta_0^{b}$ depends on the $1$-norm of $b$.  It exists a relationship between the probability
$\mathrm{P}_1=\dots=\mathrm{P}_q=\overline{\mathrm{P}}$ and the
probability of each subset of $\CC$  when the value of the the
$1$-norm of $b$  is fixed. The unconditional efficiency expected
score $E_{\overline \UU}(\Theta_0^{})$ is the weighted average of efficiency scores  where the weight
attached to each of them  is given by the number of all  specifications with the
same number of inputs/outputs.

\subsection{Summing up}

To end this section, we summarize the results obtained.
We have proposed a probabilistic/combinatorial  approach that provides DEA efficiency scores which are robust against the selection of inputs/outputs to be included in the model. This
approach considers all the scenarios associated with the possible specifications of the DEA model together with their probabilities. The key is obviously
 in the specification of such probabilities. If they can be set by using expert opinions, then the DEA efficiency scores will be obtained as
 in (\ref{esperanza1.1}). Otherwise, if information from experts is not available, then an interesting
 choice is $\frac{1}{2^{q}}  \sum_{b\in \{0,1\}^{q}} \theta_0^b$, that is, the average of DEA efficiency scores across all the scenarios.
 The developments have shown that this score is that which results both when $p$ is set by maximizing the entropy and when a Bayesian approach
  is adopted by assuming the probabilities of selection of candidates as independent random variables distributed uniform in $U[0,1]$. In particular,
  the entropy is maximized when the probabilities of selection of candidates are all equal to  one half. Having a look at (\ref{col4}),
  we realize that $E_{\overline\UU}(\Theta_0^{})$ somehow summarizes the values $E_{\overline{\mathrm{p}}}(\Theta_0 )$ across the different values
   of a common $\overline{\mathrm{p}}$ (with $\overline{p}=1/2$ among them).

\section{Algorithm}

In order to compute the   efficiency scores   we have defined  for 
each subset of inputs/outputs $b\in\{0,1\}^q$ and for each DMU, we need to  solve  $(2^{q}-1)\cdot n
$ problems, where $n$ is the number of DMUs to evaluate. To  represent the
subsets of  $\CC$, we consider that $\CC$ is ordered and  use a
list of $q$ binary numbers. If $q=4,$ the number 1001 represents
the subset with the first and the last elements of the ordered set
$\CC.$

Figure \ref{tree}  shows the exploration tree, where each node in
the tree represents one of the  $2^{q}=16$  problems, and  only 4
inputs/outputs have been considered. Each node  is enumerated
according to the order of exploration.   For example, node 12 will
be explored after node 11 and corresponds to the computation of
the score  when $b=\{0101\}$.

\begin{figure}[h]
\begin{center}
   \includegraphics{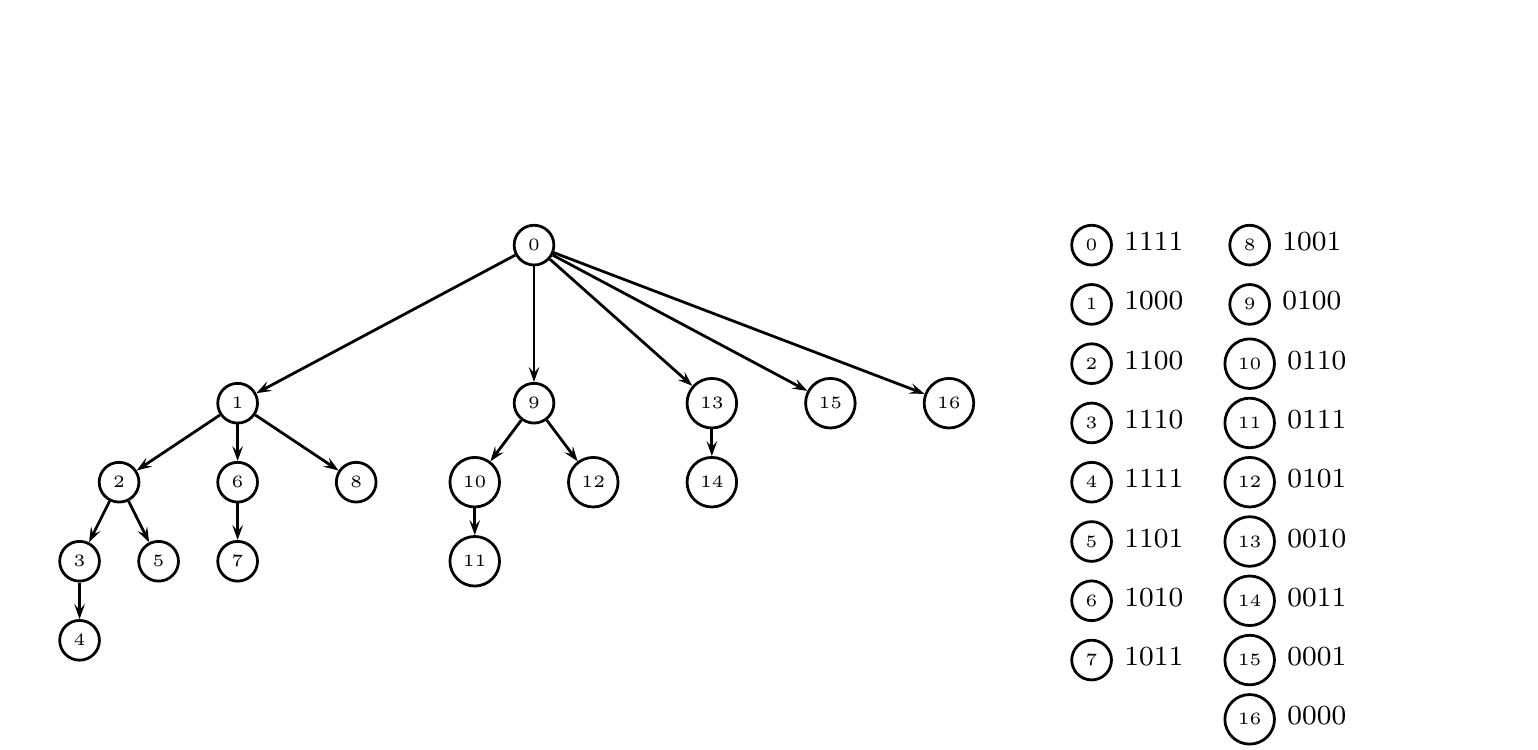}
\caption{Enumeration tree.} \label{tree}
\end{center}
\end{figure}

\begin{figure}[htb]
\centering
\fbox{
\begin{minipage}{13.cm}
\begin{minipage}{12.5cm}
\vspace{.5cm}
\footnotesize
\begin{enumerate}
\item Solve the problem (\ref{CRS}) when all the inputs/outputs
are considered. Let $\theta_0^{b=(1,\cdots,1)}$ be the solution of
the problem.
    \begin{itemize}
    \item[1.1 -] $\theta_0^{C}=\theta_0^{b=(1,\cdots,1)}$ is un upper bound of   $\theta_0^{b}$ for all $b \subset \{0,1\}^q$.
    \item[1.2 -] Get the dual variables of the problem in the optimal solution. It is  known that dual variables provide the weights of the inputs and outputs that represents the   $DMU_0$ in  the best way. i.e., dual variables measure the importance of each input/output for  $DMU_0$.
    \item[1.3 -] Order the  set $\CC$ according  the dual variables, the highest values first.
    \end{itemize}
\item Start the algorithm at node 1. Solve the problem with the
first input/output of $\CC$, i.e.,  $b=(1,0,\cdots,0)$. \item At
each node $b$ in the enumeration tree
    \begin{itemize}
    \item[3.1 -] Add the solution of the ancestor node as a basic initial solution.
    \item[3.2 -] Solve (using the dual simplex algorithm) the problem (\ref{CRS}) for the inputs/outputs in the set $b \subset \{0,1\}^q$.
    \end{itemize}

\item If the last element of $\CC$ belongs to set $b$, then,  cut
the branch in the enumeration tree, and go  back in the tree to
the first node that allows to add a new element of $\CC$ and
provide a new subset $b$ not evaluated yet. \item
    \begin{itemize}
        \item[5.1-]     If $\theta_0^{b} = \theta_0^{b=(1,\cdots,1)}$, then
        \begin{itemize}
        \item[a) ]  Cut the branch in the enumeration tree, and go  back in the tree to the first node that allows us to add a new element of $\CC$ and provide a new subset $b$ not evaluated yet.
        \item[b) ]  $\theta_0^{b'} := \theta_0^{b=(1,\cdots,1)}$  for each subset/problem $b'$, with $b \subset  b'$.
        \item[c) ] Go to step 3.
        \end{itemize}
    \item[5.2 -] else if $\theta_0^{b} \neq  \theta_0^{b=(1,\cdots,1)}$
        \begin{itemize}
        \item[a) ] Add a first element of  ordered set $\CC$ that provide a new subset $b$  not evaluated yet.
        \item[b) ] Solve the new problem to obtain $\theta_0^{b} $  and the dual variables of problem (\ref{CRS}) for the set of inputs/outputs $b$.
            \item[c) ]Go to step 3.
        \end{itemize}
    \end{itemize}

\end{enumerate}
\vspace{.3cm}
 \end{minipage}
 \end{minipage}
} \label{algorithm2} \caption{Enumeration  tree algorithm.}
\end{figure}

In order to reduce the number of problems to calculate, we propose the following enumeration algorithm for each $DMU_0$, see Figure 2. The main contributions of the algorithm are summarized in three points:

\begin{enumerate}
    \item The  set of inputs/outputs $\CC$ is sorted in such a way that the scores of the first nodes are  hopefully large. If an score equals the upper bound given by $\theta_0^{C}=\theta_0^{b=(1,\cdots,1)},$ then 
     it is not necessary
     to continue the exploration.     Step 1.3 of the algorithm.
    \item At each iteration, the solver uses the optimal solution of the ancestor node in the tree as a feasible solution.  Each solution in the exploration tree is a feasible dual solution for the next problem.
     Running the dual simplex algorithm  with start basic solution allows to reduce the computational time.  Step 3 of the algorithm.
    \item At each iteration, the resulting score is compared  with the score from set  $\CC.$ If both are equal, the exploration tree is cut. Step 5.1 of the algorithm.
\end{enumerate}

We have tested the algorithm using randomly generated instances. We
have used a uniform distribution in $[50,100]$ to generate the
inputs/outputs values. The results are shown in Table
\ref{t:tabla1}. The first  two  columns  are the number of DMUs
($N_{DMU}$), and  the number of  inputs/outputs ($|\CC|$). The number
of DMUs and inputs/outputs varies from 25 to 200, and from 5 to 20
respectively. The next block of two columns gives the related
elapsed time ($time$) in seconds and the number of solved problems
by the algorithm   ($N.$ $Prob$). The next two columns report the
reduction in time (\%$time$) and the reduction in the number of
problems (\%$N.$ $Prob$) comparing the algorithm with the full
exhaustive exploration of the tree. Values for the exhaustive
exploration  are reported in the last two columns. Note that the
number of problems to solve  is $(2^{|\CC|}-1)*N_{DMU}$
($(2^5-1)*25=775, (2^{10}-1)*25=25575,\ldots$). In general, the
use of the algorithm instead of the exhaustive enumeration
divides the time by two and avoids to solve more than half of
problems. Nevertheless, we can see in Table
\ref{t:tabla1} that computational time required by our approach is very low when the number of input/output candidates is no larger than 10. This means that, in many of the situations we usually find in practice, computational time will not be a limitation in order to yield robust DEA efficiency scores (and, consequently, in those cases, the use of the algorithm will not provide any important benefit in terms of time).

Our experiment was conducted on a PC with a 2.5 GHz dual-core Intel Core i5 processor, 8 Gb of RAM and the operating system was OS X 10.9.

\setlength{\tabcolsep}{1.6 mm}

 \begin{table}[htb]
 \centering
 \footnotesize
\begin{tabular}{llrrrrrrrrr}
\hline  \multicolumn{2}{l}{} && \multicolumn{2}{c}{$Algorithm$ } &&\multicolumn{2}{c}{$Reduction$} &&\multicolumn{2}{c}{$Total$}
\\ \cline{4-5}\cline{7-8}\cline{10-11}
 $N_{DMU}$ &$|\CC|$ && $time$ & $N. Prob$  & & $\%time$ & $\%N. Prob$ && $time$ &$N. Prob$     \\ \hline
\\ \cline{1-2}
25  &   5   &&  0.01    &   587 &&  50\%    &   76\%    &&  0.02    &   775 \\
25  &   10  &&  0.21    &   16,529  &&  60\%    &   65\%    &&  0.35    &   25,575  \\
25  &   15  &&  5.75    &   397,900 &&  52\%    &   49\%    &&  11.08   &   819,175 \\
25  &   20  &&  132.00  &   7,190,698   &&  42\%    &   27\%    &&  317.47  &   26,214,375  \\
\\ \cline{1-2}
50  &   5   &&  0.03    &   1,262   &&  75\%    &   81\%    &&  0.04    &   1,550   \\
50  &   10  &&  0.51    &   38,923  &&  55\%    &   76\%    &&  0.92    &   51,150  \\
50  &   15  &&  16.30   &   1,050,132   &&  48\%    &   64\%    &&  33.97   &   1,638,350   \\
50  &   20  &&  400.63  &   22,374,653  &&  44\%    &   43\%    &&  901.38  &   52,428,750  \\
\\ \cline{1-2}
100 &   5   &&  0.05    &   2,012   &&  63\%    &   65\%    &&  0.08    &   3,100   \\
100 &   10  &&  1.51    &   82,954  &&  55\%    &   81\%    &&  2.74    &   102,300 \\
100 &   15  &&  46.15   &   2,175,660   &&  49\%    &   66\%    &&  94.12   &   3,276,700   \\
100 &   20  &&  1,095.13    &   52,321,522  &&  44\%    &   50\%    &&  2,480.91    &   104,857,500 \\
\\ \cline{1-2}
200 &   5   &&  0.14    &   4,461   &&  67\%    &   72\%    &&  0.21    &   6,200   \\
200 &   10  &&  4.41    &   167,624 &&  53\%    &   82\%    &&  8.26    &   204,600 \\
200 &   15  &&  139.27  &   4,553,112   &&  49\%    &   69\%    &&  285.63  &   6,553,400   \\
200 &   20  &&  4,498.19    &   108,572,383 &&  49\%    &   52\%    &&  9,118.91    &   209,715,000 \\
\hline
\end{tabular}
\caption{Computational results for a big random data.}
\label{t:tabla1}
\end{table}

\newpage

\section{Combinatorial DEA Efficiency Scores: A case study}

To illustrate the approach proposed, we apply it in an example on the assessment 
of profesional  tennis players. The
Association of Tennis Professionals (ATP) provides statistics of the
game performance of players, in particular regarding the game
factors reported in Table \ref{t:tabla2}. Table \ref{t:tabla3}
records the values of those game factors corresponding to the 46 players for whom we have available data for all these factors during the 2014 season
(the data have taken from \url{http://www.atpworldtour.com/}). With
these 46 players, we  have a sample size large enough so as to avoid
problems with the dimensionality of the models used. These game
factors are considered as outputs in the DEA efficiency models used,
while we do not consider any explicit inputs, since in our analysis
there is no reference to resources consumed. We only include in the
models a single constant input equal to 1, which means that every
player is doing the best for playing his game. Thus, we will be
actually evaluating the effectiveness of player game performance
instead of efficiency. It should also be noted that, in the case of
having one constant input, the optimal solutions of (\ref{CRS})
satisfy the condition $\sum_{j=1}^n \lambda_j =1$. Therefore, in
these special circumstances, the specification of returns to scale
is not particularly relevant (see \cite{Lovell1995,Lovell1999} for
details and discussions).

\begin{table}[h]
\centering
\small
\begin{tabular}{ll}\hline
\multicolumn{2}{l}{OUTPUTS} \\
\hline
$y_1$ &=   percentage of 1st serve\\
$y_2$&=  percentage of 1st serve points won\\
$y_3$&=  percentage of 2nd serve points won\\
$y_4$&= percentage of service games won\\
$y_5$& = percentage of break points saved\\
$y_6$&= percentage of points won returning 1st serve\\
$y_7$&= percentage of points won returning 2nd serve\\
$y_8$&= percentage of break points converted\\
$y_9$& = percentage of return games won\\ \hline
\end{tabular}
\caption{Output summary.}
\label{t:tabla2}
\end{table}


\begin{table}[htb]
\centering
\scriptsize
\begin{tabular}{rlrrrrrrrrr}\hline
ATP && \multicolumn{9}{c}{OUTPUTS} \\
\cline{3-11}
Ranking &Player &   $y_1$ &     $y_2$ & $y_3$ & $y_4$ & $y_5$ & $y_6$ & $y_7$ & $y_8$ & $y_9$ \\ \hline  \hline
1   &   Novak Djokovic      &   67  &   75  &   56  &   88  &   63  &   33  &   58  &   45  &   33  \\
2   &   Roger Federer   &   64  &   79  &   58  &   91  &   71  &   32  &   51  &   39  &   26  \\
3   &   Rafael Nadal        &   70  &   72  &   55  &   85  &   66  &   35  &   56  &   48  &   35  \\
4   &   Stan Wawrinka       &   55  &   79  &   54  &   86  &   61  &   29  &   50  &   42  &   22  \\
5   &   Kei Nishikori       &   60  &   73  &   54  &   84  &   64  &   30  &   53  &   42  &   28  \\
6   &   Andy Murray     &   60  &   74  &   51  &   81  &   61  &   33  &   55  &   44  &   32  \\
7   &   Tomas Berdych       &   58  &   78  &   54  &   86  &   63  &   30  &   54  &   39  &   25  \\
8   &   Milos Raonic        &   61  &   83  &   54  &   90  &   69  &   27  &   45  &   39  &   16  \\
9   &   Marin Cilic     &   57  &   79  &   50  &   85  &   66  &   30  &   50  &   37  &   22  \\
10  &   David Ferrer        &   63  &   69  &   52  &   79  &   62  &   34  &   56  &   43  &   33  \\
11  &   Grigor Dimitrov     &   61  &   77  &   54  &   86  &   64  &   29  &   50  &   42  &   22  \\
12  &   Jo-Wilfried Tsonga      &   62  &   77  &   54  &   87  &   70  &   29  &   45  &   39  &   18  \\
13  &   Ernests Gulbis      &   60  &   78  &   51  &   85  &   64  &   29  &   50  &   40  &   21  \\
14  &   Feliciano Lopez     &   60  &   78  &   53  &   86  &   69  &   25  &   45  &   34  &   15  \\
15  &   Roberto Bautista Agut   &   69  &   70  &   53  &   81  &   65  &   31  &   53  &   43  &   26  \\
16  &   Kevin Anderson      &   66  &   75  &   51  &   86  &   69  &   26  &   48  &   35  &   18  \\
17  &   Tommy Robredo       &   64  &   74  &   54  &   85  &   64  &   29  &   49  &   37  &   21  \\
18  &   Gael Monfils        &   65  &   73  &   50  &   82  &   62  &   34  &   50  &   40  &   27  \\
19  &   John Isner      &   68  &   79  &   57  &   93  &   75  &   24  &   42  &   24  &   9   \\
20  &   Fabio Fognini       &   59  &   69  &   48  &   73  &   56  &   32  &   51  &   43  &   27  \\
21  &   Gilles Simon        &   56  &   71  &   51  &   78  &   63  &   31  &   53  &   45  &   26  \\
23  &   Alexandr Dolgopolov &   55  &   75  &   52  &   82  &   62  &   30  &   51  &   38  &   23  \\
24  &   Philipp Kohlschreiber   &   59  &   73  &   56  &   84  &   62  &   29  &   50  &   43  &   23  \\
25  &   Julien Benneteau    &   64  &   72  &   50  &   82  &   63  &   28  &   49  &   37  &   21  \\
26  &   Richard Gasquet     &   65  &   73  &   56  &   84  &   59  &   28  &   50  &   38  &   21  \\
28  &   Leonardo Mayer      &   60  &   75  &   54  &   82  &   61  &   29  &   49  &   38  &   19  \\
29  &   Jeremy Chardy   &       59  &   77  &   50  &   82  &   63  &   27  &   50  &   39  &   19  \\
31  &   Lukas Rosol &       57  &   72  &   50  &   78  &   60  &   27  &   47  &   41  &   17  \\
32  &   Santiago Giraldo        &   63  &   70  &   50  &   78  &   63  &   30  &   49  &   42  &   23  \\
33  &   Fernando Verdasco       &   66  &   72  &   51  &   82  &   66  &   30  &   49  &   39  &   22  \\
35  &   Sam Querrey &       61  &   79  &   52  &   87  &   67  &   26  &   46  &   37  &   16  \\
36  &   Guillermo Garcia-Lopez  &   57  &   69  &   48  &   74  &   58  &   32  &   50  &   40  &   26  \\
38  &   Yen-Hsun Lu &       64  &   71  &   52  &   80  &   66  &   26  &   51  &   41  &   21  \\
39  &   Dominic Thiem   &       58  &   71  &   51  &   77  &   60  &   29  &   50  &   37  &   22  \\
40  &   Benjamin Becker     &   59  &   74  &   49  &   79  &   60  &   27  &   49  &   37  &   20  \\
41  &   Pablo Andujar   &       66  &   64  &   52  &   73  &   57  &   32  &   53  &   42  &   29  \\
42  &   Jack Sock   &       59  &   76  &   54  &   86  &   69  &   25  &   47  &   42  &   19  \\
43  &   Jerzy Janowicz      &   60  &   74  &   47  &   79  &   60  &   27  &   48  &   39  &   18  \\
45  &   Andreas Seppi       &   57  &   70  &   52  &   77  &   64  &   31  &   49  &   37  &   22  \\
46  &   Marcel Granollers       &   61  &   69  &   48  &   73  &   54  &   29  &   49  &   40  &   23  \\
49  &   Denis Istomin   &       68  &   72  &   51  &   81  &   63  &   27  &   48  &   40  &   19  \\
54  &   Joao Sousa  &       60  &   67  &   48  &   72  &   61  &   29  &   46  &   38  &   19  \\
60  &   Federico Delbonis       &   62  &   71  &   53  &   81  &   63  &   27  &   49  &   36  &   18  \\
73  &   Jarkko Nieminen     &   65  &   70  &   49  &   77  &   58  &   26  &   52  &   38  &   20  \\
75  &   Marinko Matosevic       &   59  &   71  &   49  &   76  &   63  &   28  &   50  &   38  &   20  \\
87  &   Edouard Roger-Vasselin      &   67  &   70  &   51  &   82  &   68  &   26  &   48  &   37  &   18  \\
\hline
\end{tabular}
\caption{Data (Source: \emph{http://www.atpworldtour.com/}).}
\label{t:tabla3}
\end{table}

We present the results obtained by distinguishing between whether information from experts is available or not.

Suppose that some experts are asked to give their opinion about
whether or not to include each of the factors listed in Table
\ref{t:tabla2} in a DEA model aimed at providing measures of
effectiveness of the game performance of players. We might have observed  that
all the experts agree in considering $y_4$ and $y_9$, 80\% out of
them would include $y_2$, $y_3$, $y_5$, $y_6$, $y_7$ and $y_8$,
while only 40\% would be in favor of incorporating $y_1$. This would
be showing that the most important game factor in the opinion of
experts are those that have to do with winning games ($y_4$ and
$y_9$), whereas $y_1$, which does not have a direct influence on the
result of the matches, is viewed as a less relevant game factor.
With such information from experts, the probabilities of selection
of candidates can be estimed as follows: $p_1=0.4$, $p_2=0.8$,
$p_3=0.8$, $p_4=1$, $p_5=0.8$, $p_6=0.8$, $p_7=0.8$, $p_8=0.8$ and
$p_9=1$. The efficiency scores (\ref{esperanza1.1}) associated with
that specification of the probabilities (together with the corresponding standard deviations) are recorded in Table
\ref{t:tabla5}.

\setlength{\tabcolsep}{4.1 mm}

\begin{table}[htb]
\centering
\scriptsize
\begin{tabular}{rrllrrr}\hline
&&\multicolumn{5}{c}{$p_1=0.4$;\, $p_2=p_3=p_5=p_6=p_7=p_8=0.8$;\,  $p_4=p_9=1.0$  }  \\  \cline{3-7}
ATP	&&	&&$E_p(\Theta_0)$&&$\sqrt{V_p(\Theta_0)}$\\ \hline
1	&&	Novak	Djokovic	&&	1.00000	&&	0.00000		\\	
2	&&	Roger	Federer	&&	1.00000	&&	0.00000		\\	
3	&&	Rafael	Nadal	&&	1.00000	&&	0.00000		\\	
4	&&	Stan	Wawrinka	&&	0.99003	&&	0.01369		\\	
5	&&	Kei	Nishikori	&&	0.96147	&&	0.00334		\\	
6	&&	Andy	Murray	&&	0.98238	&&	0.01122		\\	
7	&&	Tomas	Berdych	&&	0.99052	&&	0.01435		\\	
8	&&	Milos	Raonic	&&	0.99764	&&	0.00532		\\	
9	&&	Marin	Cilic	&&	0.98129	&&	0.01760	\\	
10	&&	David	Ferrer	&&	0.98289	&&	0.00957		\\	
11	&&	Grigor	Dimitrov	&&	0.98108	&&	0.00579		\\	
12	&&	Jo-Wilfried	Tsonga	&&	0.98172	&&	0.01036		\\	
13	&&	Ernests	Gulbis	&&	0.97374	&&	0.01300		\\	
14	&&	Feliciano	Lopez	&&	0.96362	&&	0.00744		\\	
15	&&	Roberto	Bautista Agut	&&	0.97068	&&	0.01465		\\	
16	&&	Kevin	Anderson	&&	0.96617	&&	0.00988		\\	
17	&&	Tommy	Robredo	&&	0.94801	&&	0.00922		\\	
18	&&	Gael	Monfils	&&	0.97947	&&	0.01824		\\	
19	&&	John	Isner	&&	1.00000	&&	0.00000		\\	
20	&&	Fabio	Fognini	&&	0.92927	&&	0.01139		\\	
21	&&	Gilles	Simon	&&	0.95820	&&	0.01039		\\	
23	&&	Alexandr	Dolgopolov	&&	0.95108	&&	0.01398		\\	
24	&&	Philipp	Kohlschreiber	&&	0.98006	&&	0.01537		\\	
25	&&	Julien	Benneteau	&&	0.92934	&&	0.01280		\\	
26	&&	Richard	Gasquet	&&	0.96783	&&	0.01383		\\	
28	&&	Leonardo	Mayer	&&	0.94988	&&	0.00270		\\	
29	&&	Jeremy	Chardy	&&	0.95884	&&	0.01904		\\	
31	&&	Lukas	Rosol	&&	0.92203	&&	0.01788		\\	
32	&&	Santiago	Giraldo	&&	0.92950	&&	0.00862		\\	
33	&&	Fernando	Verdasco	&&	0.94870	&&	0.01137		\\	
35	&&	Sam	Querrey	&&	0.96701	&&	0.00704		\\	
36	&&	Guillermo	Garcia-Lopez	&&	0.92441	&&	0.01561		\\	
38	&&	Yen-Hsun	Lu	&&	0.95274	&&	0.01247		\\	
39	&&	Dominic	Thiem	&&	0.91262	&&	0.00916		\\	
40	&&	Benjamin	Becker	&&	0.92654	&&	0.01875		\\	
41	&&	Pablo	Andujar	&&	0.93772	&&	0.00806		\\	
42	&&	Jack	Sock	&&	0.98524	&&	0.01606		\\	
43	&&	Jerzy	Janowicz	&&	0.92598	&&	0.01701		\\	
45	&&	Andreas	Seppi	&&	0.92717	&&	0.00683		\\	
46	&&	Marcel	Granollers	&&	0.89783	&&	0.01687		\\	
49	&&	Denis	Istomin	&&	0.94533	&&	0.02649		\\	
54	&&	Joao	Sousa	&&	0.88430	&&	0.00786		\\	
60	&&	Federico	Delbonis	&&	0.92443	&&	0.00748		\\	
73	&&	Jarkko	Nieminen	&&	0.92289	&&	0.02003		\\	
75	&&	Marinko	Matosevic	&&	0.92032	&&	0.01334		\\	
87	&&	Edouard	Roger-Vasselin	&&	0.95572	&&	0.01918		\\	\hline
\end{tabular}
\caption{Robust efficiency scores reflecting expert opinion.}
\label{t:tabla5}
\end{table}

We can see that the top 3 players in the ATP ranking also achieve the maximum rating when evaluated with the robust DEA efficiency scores: Djokovic, Federer and Nadal. Note, however, that Isner, which is also scored 1, occupied the  $19^{th}$ position in the ATP ranking. This shows that the assessment of Isner is better when game performance instead of competitive performance is evaluated, for this specification of $p$. In contrast, this analysis reveals that other players with lower efficiency scores, like Nishikori (0.961), perform better in competition (he ranked $5^{th}$).

As discussed in Section 4, if information from experts is not
available, the scores $E^e(\Theta_0)$, which coincide with
$E_{\UU}(\Theta_0)$, and the scores $E_{\overline\UU}(\Theta_0^{})$ may yield
useful information. These two scores are reported in Table
\ref{t:tabla6}, together with the corresponding standard deviations (which are somewhat larger than the standard deviations of the scores in Table
\ref{t:tabla5}.

\setlength{\tabcolsep}{2.6 mm}

\begin{table}[htb]
\centering
\scriptsize
\begin{tabular}{rlrrrrrrr}\hline
ATP && \multicolumn{6}{c}{} \\
Ranking &Player &   &\multicolumn{1}{c}{$E_{\overline\UU}(\Theta_0^{})$} &\multicolumn{1}{c}{$\sqrt{V_{\overline\UU}(\Theta_0^{})}$} & & \multicolumn{1}{c}{$E^e(\Theta_0)$}&\multicolumn{1}{c}{$\sqrt{V^e(\Theta_0^{})}$}\\ \hline
1	&	Novak	Djokovic	&	&	0.98992	&	0.02528	&	&	0.99467	&	0.01571		\\
2	&	Roger	Federer	&	&	0.98550	&	0.04209	&	&	0.99302	&	0.02649		\\
3	&	Rafael	Nadal	&	&	0.99340	&	0.02364	&	&	0.99739	&	0.01372		\\
4	&	Stan	Wawrinka	&	&	0.96032	&	0.06051	&	&	0.96259	&	0.04108		\\
5	&	Kei	Nishikori	&	&	0.95686	&	0.02159	&	&	0.95371	&	0.01229		\\
6	&	Andy	Murray	&	&	0.96392	&	0.03703	&	&	0.96333	&	0.02641		\\
7	&	Tomas	Berdych	&	&	0.96464	&	0.05112	&	&	0.96600	&	0.03437		\\
8	&	Milos	Raonic	&	&	0.96831	&	0.07563	&	&	0.97674	&	0.04701		\\
9	&	Marin	Cilic	&	&	0.95207	&	0.06091	&	&	0.95242	&	0.04247		\\
10	&	David	Ferrer	&	&	0.96641	&	0.03692	&	&	0.96622	&	0.02542		\\
11	&	Grigor	Dimitrov	&	&	0.97640	&	0.01520	&	&	0.97442	&	0.01080		\\
12	&	Jo-Wilfried	Tsonga	&	&	0.95713	&	0.06473	&	&	0.96110	&	0.03774		\\
13	&	Ernests	Gulbis	&	&	0.94822	&	0.05713	&	&	0.94836	&	0.03728		\\
14	&	Feliciano	Lopez	&	&	0.93799	&	0.07853	&	&	0.94292	&	0.04579		\\
15	&	Roberto	Bautista Agut	&	&	0.96293	&	0.04204	&	&	0.96425	&	0.02814		\\
16	&	Kevin	Anderson	&	&	0.94745	&	0.06841	&	&	0.95185	&	0.03969		\\
17	&	Tommy	Robredo	&	&	0.93902	&	0.05284	&	&	0.93975	&	0.02761		\\
18	&	Gael	Monfils	&	&	0.95848	&	0.04582	&	&	0.95741	&	0.03163		\\
19	&	John	Isner	&	&	0.96937	&	0.11114	&	&	0.98667	&	0.06277		\\
20	&	Fabio	Fognini	&	&	0.91579	&	0.04771	&	&	0.91055	&	0.02786		\\
21	&	Gilles	Simon	&	&	0.94054	&	0.04487	&	&	0.93859	&	0.02809		\\
23	&	Alexandr	Dolgopolov	&	&	0.93136	&	0.05270	&	&	0.92907	&	0.03143		\\
24	&	Philipp	Kohlschreiber	&	&	0.95484	&	0.05378	&	&	0.95472	&	0.03545		\\
25	&	Julien	Benneteau	&	&	0.92161	&	0.05456	&	&	0.91868	&	0.02878		\\
26	&	Richard	Gasquet	&	&	0.94954	&	0.05970	&	&	0.95132	&	0.03656		\\
28	&	Leonardo	Mayer	&	&	0.94966	&	0.01981	&	&	0.94546	&	0.00802		\\
29	&	Jeremy	Chardy	&	&	0.93265	&	0.06401	&	&	0.93022	&	0.04035		\\
31	&	Lukas	Rosol	&	&	0.90260	&	0.06551	&	&	0.89649	&	0.03520		\\
32	&	Santiago	Giraldo	&	&	0.91925	&	0.04587	&	&	0.91580	&	0.02293		\\
33	&	Fernando	Verdasco	&	&	0.93944	&	0.04965	&	&	0.93934	&	0.02768		\\
35	&	Sam	Querrey	&	&	0.94287	&	0.07236	&	&	0.94693	&	0.04235		\\
36	&	Guillermo	Garcia-Lopez	&	&	0.90890	&	0.05076	&	&	0.90174	&	0.03105		\\
38	&	Yen-Hsun	Lu	&	&	0.93516	&	0.05439	&	&	0.93480	&	0.03005		\\
39	&	Dominic	Thiem	&	&	0.90117	&	0.05250	&	&	0.89569	&	0.02571		\\
40	&	Benjamin	Becker	&	&	0.90732	&	0.06147	&	&	0.90126	&	0.03619		\\
41	&	Pablo	Andujar	&	&	0.93193	&	0.04401	&	&	0.93072	&	0.02510		\\
42	&	Jack	Sock	&	&	0.95629	&	0.06702	&	&	0.95888	&	0.04094		\\
43	&	Jerzy	Janowicz	&	&	0.90613	&	0.06477	&	&	0.90113	&	0.03790		\\
45	&	Andreas	Seppi	&	&	0.91392	&	0.05198	&	&	0.91156	&	0.02741		\\
46	&	Marcel	Granollers	&	&	0.88992	&	0.05478	&	&	0.88115	&	0.02775		\\
49	&	Denis	Istomin	&	&	0.94078	&	0.06500	&	&	0.94003	&	0.04246		\\
54	&	Joao	Sousa	&	&	0.87863	&	0.05936	&	&	0.87207	&	0.02565		\\
60	&	Federico	Delbonis	&	&	0.91450	&	0.06086	&	&	0.91361	&	0.02985		\\
73	&	Jarkko	Nieminen	&	&	0.91497	&	0.06265	&	&	0.91095	&	0.03712		\\
75	&	Marinko	Matosevic	&	&	0.90415	&	0.05746	&	&	0.89856	&	0.03113		\\
87	&	Edouard	Roger-Vasselin	&	&	0.93960	&	0.06824	&	&	0.94121	&	0.04184		\\\hline
\end{tabular}
\caption{Unconditional expected score  $E_{\overline\UU}(\Theta_0^{})$  and the maximum entropy expected score $E^e(\Theta_0)$.}
\label{t:tabla6}
\end{table}

As said before, $E_{\overline\UU}(\Theta_0^{})$ somehow summarizes the values $E_{\overline{\mathrm{p}}}(\Theta_0)$ across the different values of a common $\overline{\mathrm{p}}$. Table \ref{t:tabla4} reports the scores $E_{\overline{\mathrm{p}}}(\Theta_0)$ for some values of $\overline{\mathrm{p}}$. This information can complement that provided by Table \ref{t:tabla6}. Note, in particular, that the scores $E^e(\Theta_0)$ are actually the ones under column $\overline{\mathrm{p}}=0.5$ in Table \ref{t:tabla4}.

\setlength{\tabcolsep}{1.4 mm}

\begin{table}[htb]
\tiny
\begin{tabular}{rlrrrrrrrrrr}\hline
ATP && \multicolumn{10}{c} { $\overline{\mathrm{p}}$ }\\
\cline{3-12}
    &       &   0.1 &   0.2 &   0.3 &   0.4 &   0.5 &   0.6 &   0.7 &   0.8 &   0.9 &   1   \\ \hline
1   &   Novak Djokovic  &   0.97099 &   0.97167 &   0.98048 &   0.98892 &   0.99467 &   0.99787 &   0.99934 &   0.99987 &   0.99999 &   1.00000 \\
2   &   Roger Federer   &   0.95663 &   0.95867 &   0.97238 &   0.98488 &   0.99302 &   0.99732 &   0.99921 &   0.99985 &   0.99999 &   1.00000 \\
3   &   Rafael Nadal    &   0.97906 &   0.98069 &   0.98770 &   0.99373 &   0.99739 &   0.99915 &   0.99980 &   0.99997 &   1.00000 &   1.00000 \\
4   &   Stan Wawrinka   &   0.92217 &   0.91426 &   0.92913 &   0.94732 &   0.96259 &   0.97410 &   0.98263 &   0.98914 &   0.99431 &   0.99865 \\
5   &   Kei Nishikori   &   0.95726 &   0.94561 &   0.94594 &   0.94972 &   0.95371 &   0.95699 &   0.95948 &   0.96131 &   0.96259 &   0.96340 \\
6   &   Andy Murray &   0.94381 &   0.93456 &   0.94210 &   0.95327 &   0.96333 &   0.97124 &   0.97726 &   0.98188 &   0.98557 &   0.98862 \\
7   &   Tomas Berdych   &   0.93177 &   0.92443 &   0.93711 &   0.95279 &   0.96600 &   0.97605 &   0.98370 &   0.98984 &   0.99514 &   1.00000 \\
8   &   Milos Raonic    &   0.92162 &   0.92009 &   0.94043 &   0.96128 &   0.97674 &   0.98680 &   0.99297 &   0.99664 &   0.99878 &   1.00000 \\
9   &   Marin Cilic &   0.91717 &   0.90617 &   0.91945 &   0.93708 &   0.95242 &   0.96435 &   0.97346 &   0.98062 &   0.98651 &   0.99163 \\
10  &   David Ferrer    &   0.94732 &   0.93955 &   0.94704 &   0.95731 &   0.96622 &   0.97310 &   0.97836 &   0.98253 &   0.98599 &   0.98901 \\
11  &   Grigor Dimitrov &   0.97395 &   0.96739 &   0.96829 &   0.97130 &   0.97442 &   0.97711 &   0.97932 &   0.98116 &   0.98268 &   0.98392 \\
12  &   Jo-Wilfried Tsonga  &   0.92086 &   0.91459 &   0.93019 &   0.94761 &   0.96110 &   0.97044 &   0.97689 &   0.98167 &   0.98561 &   0.98921 \\
13  &   Ernests Gulbis  &   0.91853 &   0.90697 &   0.91875 &   0.93465 &   0.94836 &   0.95886 &   0.96675 &   0.97282 &   0.97766 &   0.98164 \\
14  &   Feliciano Lopez &   0.90002 &   0.89019 &   0.90774 &   0.92773 &   0.94292 &   0.95293 &   0.95929 &   0.96347 &   0.96640 &   0.96861 \\
15  &   Roberto Bautista Agut   &   0.93961 &   0.93244 &   0.94220 &   0.95430 &   0.96425 &   0.97156 &   0.97683 &   0.98069 &   0.98357 &   0.98571 \\
16  &   Kevin Anderson  &   0.91127 &   0.90295 &   0.91914 &   0.93761 &   0.95185 &   0.96151 &   0.96785 &   0.97210 &   0.97503 &   0.97707 \\
17  &   Tommy Robredo   &   0.91825 &   0.90662 &   0.91679 &   0.92968 &   0.93975 &   0.94665 &   0.95140 &   0.95496 &   0.95793 &   0.96064 \\
18  &   Gael Monfils    &   0.93252 &   0.92273 &   0.93239 &   0.94566 &   0.95741 &   0.96684 &   0.97445 &   0.98085 &   0.98640 &   0.99132 \\
19  &   John Isner  &   0.90490 &   0.91171 &   0.94282 &   0.96985 &   0.98667 &   0.99514 &   0.99863 &   0.99975 &   0.99998 &   1.00000 \\
20  &   Fabio Fognini   &   0.90873 &   0.88681 &   0.89046 &   0.90067 &   0.91055 &   0.91846 &   0.92452 &   0.92924 &   0.93302 &   0.93612 \\
21  &   Gilles Simon    &   0.92314 &   0.90843 &   0.91570 &   0.92774 &   0.93859 &   0.94701 &   0.95334 &   0.95815 &   0.96188 &   0.96480 \\
23  &   Alexandr Dolgopolov &   0.91130 &   0.89505 &   0.90367 &   0.91722 &   0.92907 &   0.93816 &   0.94515 &   0.95084 &   0.95579 &   0.96024 \\
24  &   Philipp Kohlschreiber   &   0.92439 &   0.91443 &   0.92614 &   0.94148 &   0.95472 &   0.96505 &   0.97312 &   0.97973 &   0.98542 &   0.99061 \\
25  &   Julien Benneteau    &   0.90572 &   0.88779 &   0.89542 &   0.90789 &   0.91868 &   0.92685 &   0.93309 &   0.93819 &   0.94267 &   0.94684 \\
26  &   Richard Gasquet &   0.91743 &   0.90793 &   0.92144 &   0.93792 &   0.95132 &   0.96110 &   0.96821 &   0.97363 &   0.97799 &   0.98164 \\
28  &   Leonardo Mayer  &   0.95676 &   0.94305 &   0.94117 &   0.94304 &   0.94546 &   0.94747 &   0.94894 &   0.94998 &   0.95071 &   0.95119 \\
29  &   Jeremy Chardy   &   0.90449 &   0.88853 &   0.89950 &   0.91575 &   0.93022 &   0.94173 &   0.95088 &   0.95851 &   0.96523 &   0.97150 \\
31  &   Lukas Rosol &   0.88787 &   0.86394 &   0.87055 &   0.88400 &   0.89649 &   0.90662 &   0.91489 &   0.92201 &   0.92847 &   0.93452 \\
32  &   Santiago Giraldo    &   0.91185 &   0.89239 &   0.89693 &   0.90683 &   0.91580 &   0.92257 &   0.92739 &   0.93077 &   0.93303 &   0.93439 \\
33  &   Fernando Verdasco   &   0.91974 &   0.90702 &   0.91621 &   0.92889 &   0.93934 &   0.94689 &   0.95223 &   0.95613 &   0.95906 &   0.96128 \\
35  &   Sam Querrey &   0.90697 &   0.89753 &   0.91369 &   0.93242 &   0.94693 &   0.95677 &   0.96324 &   0.96763 &   0.97084 &   0.97340 \\
36  &   Guillermo Garcia-Lopez  &   0.90186 &   0.87721 &   0.88028 &   0.89093 &   0.90174 &   0.91080 &   0.91817 &   0.92436 &   0.92982 &   0.93492 \\
38  &   Yen-Hsun Lu &   0.91328 &   0.89939 &   0.90926 &   0.92312 &   0.93480 &   0.94345 &   0.94970 &   0.95419 &   0.95734 &   0.95933 \\
39  &   Dominic Thiem   &   0.89665 &   0.87221 &   0.87585 &   0.88611 &   0.89569 &   0.90304 &   0.90850 &   0.91267 &   0.91596 &   0.91864 \\
40  &   Benjamin Becker &   0.89213 &   0.86865 &   0.87507 &   0.88854 &   0.90126 &   0.91166 &   0.92012 &   0.92735 &   0.93385 &   0.93996 \\
41  &   Pablo Andujar   &   0.92257 &   0.90692 &   0.91260 &   0.92247 &   0.93072 &   0.93639 &   0.94000 &   0.94230 &   0.94379 &   0.94481 \\
42  &   Jack Sock   &   0.91650 &   0.90911 &   0.92521 &   0.94384 &   0.95888 &   0.96996 &   0.97826 &   0.98496 &   0.99086 &   0.99640 \\
43  &   Jerzy Janowicz  &   0.88965 &   0.86646 &   0.87375 &   0.88796 &   0.90113 &   0.91165 &   0.91986 &   0.92642 &   0.93184 &   0.93651 \\
45  &   Andreas Seppi   &   0.90325 &   0.88298 &   0.88930 &   0.90121 &   0.91156 &   0.91898 &   0.92394 &   0.92714 &   0.92915 &   0.93029 \\
46  &   Marcel Granollers   &   0.88988 &   0.86118 &   0.86243 &   0.87166 &   0.88115 &   0.88908 &   0.89563 &   0.90129 &   0.90644 &   0.91129 \\
49  &   Denis Istomin   &   0.91008 &   0.89766 &   0.91004 &   0.92626 &   0.94003 &   0.95069 &   0.95911 &   0.96615 &   0.97233 &   0.97792 \\
54  &   Joao Sousa  &   0.88095 &   0.85106 &   0.85306 &   0.86283 &   0.87207 &   0.87892 &   0.88361 &   0.88668 &   0.88856 &   0.88950 \\
60  &   Federico Delbonis   &   0.89905 &   0.88114 &   0.88999 &   0.90310 &   0.91361 &   0.92078 &   0.92551 &   0.92876 &   0.93111 &   0.93289 \\
73  &   Jarkko Nieminen &   0.89592 &   0.87568 &   0.88401 &   0.89822 &   0.91095 &   0.92099 &   0.92892 &   0.93551 &   0.94124 &   0.94648 \\
75  &   Marinko Matosevic   &   0.89378 &   0.86961 &   0.87474 &   0.88694 &   0.89856 &   0.90788 &   0.91509 &   0.92067 &   0.92497 &   0.92821 \\
87  &   Edouard Roger-Vasselin  &   0.90610 &   0.89435 &   0.90869 &   0.92659 &   0.94121 &   0.95183 &   0.95950 &   0.96525 &   0.96979 &   0.97348 \\
    \hline
\end{tabular}
\caption{$E_\mathrm{P}(\Theta_0^{\RR})$ for different values of
$\overline{\mathrm{P}}$.} \label{t:tabla4}
\end{table}

Table \ref{t:tabla6} also shows that the top 3 players of the ATP
ranking are the ones with the larger values of
$E_{\overline\UU}(\Theta_0^{})$: Djokovic (0.9899), Federer (0.9855)
and Nadal (0.9934). In order to provide a graphical display of how
the scores change as $\overline{\mathrm{p}}$ varies,  we have
depicted  graphically the values in the rows of Table
\ref{t:tabla4} corresponding to these three players (Figure
\ref{fig5:}). We can see that Nadal outperforms Djokovic and
Ferderer irrespective of the value of $\overline{\mathrm{p}}$.
Frequently, the scores $E_{\overline{\mathrm{p}}}(\Theta_0)$ are
larger with $\overline{\mathrm{p}}$. Note that a low
$\overline{\mathrm{p}}$ implies  a low probability of input/output
selection, also for the inputs/outputs that may  benefit the
player under assessment, while if $\overline{\mathrm{p}}$ is high, the best factors
for  each player will be in the DEA analysis with high probability.

\begin{figure}[t]
\begin{center}
                   \includegraphics[width=1.\textwidth]{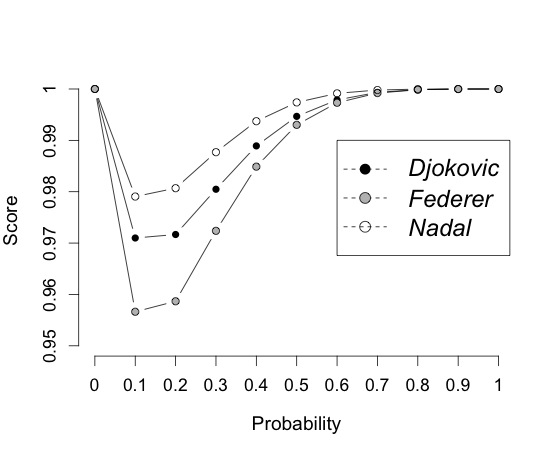}
                \caption{Djokovic vs Federer vs Nadal.}
                \label{fig5:}
                \end{center}
\end{figure}

We have also analyzed the case of Isner, in particular against that of Mayer. Figure \ref{fig7:} shows that, while the scores of Mayer remain stable for different values of
 $\overline{\mathrm{p}}$, these are quite smaller for low values of this probability. For instance, Isner, which has a very specialized game based on his service, is penalized
 when $\overline{\mathrm{p}}$ is low because $\overline{\mathrm{p}}$ is the probability of selecting  his best factor, the factor
 $y_1$.

\begin{figure}[t]
\begin{center}
                    \includegraphics[width=1.\textwidth]{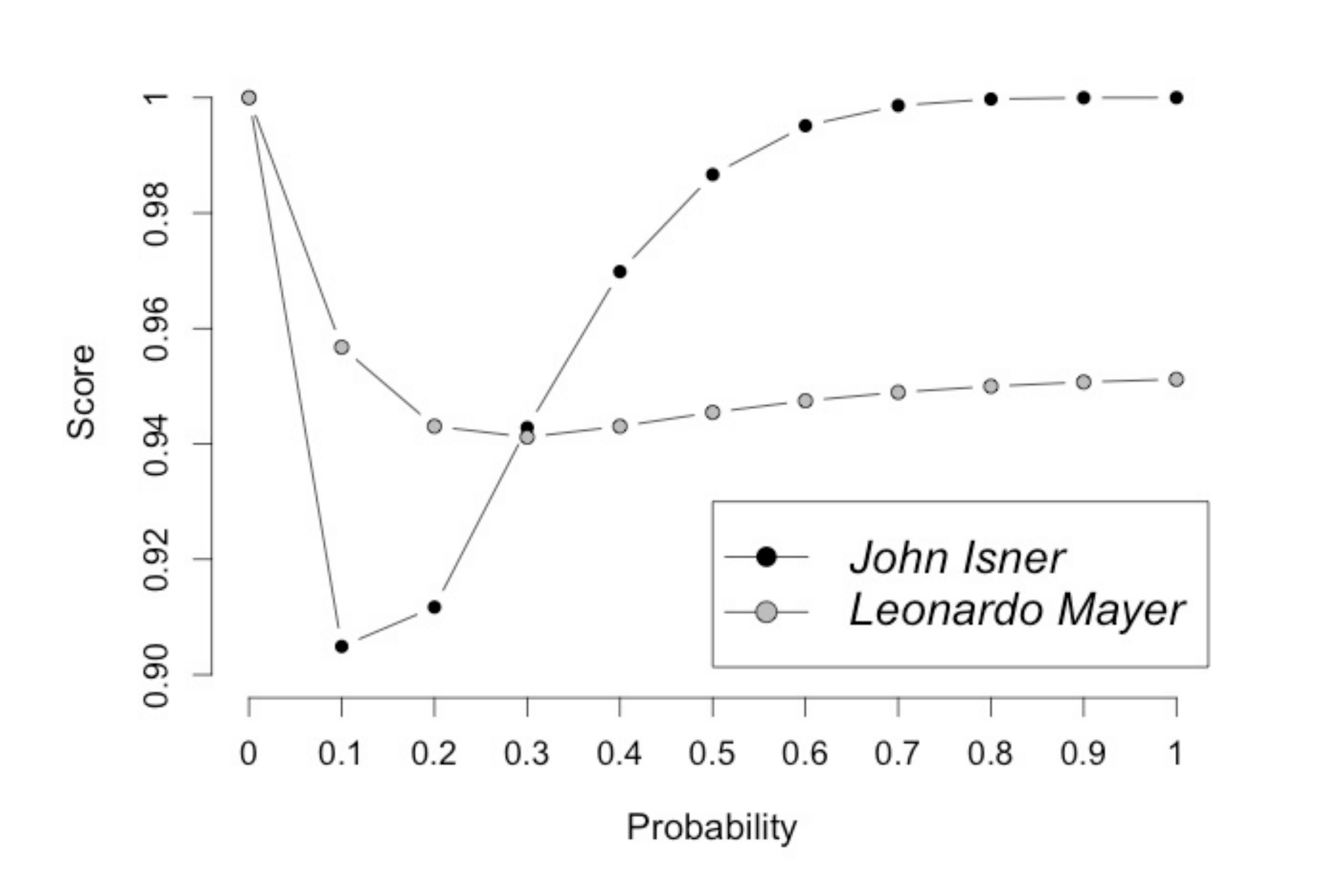}
                \caption{John Isner vs Leonardo Mayer.}
                \label{fig7:}
                \end{center}
\end{figure}

\section{Concluding Remarks}

We have proposed a probabilistic/combinatorial approach for the assessment of DMUs in DEA by using efficiency scores which are robust against the selection of inputs/outputs. Robust efficient scores are defined on the basis of the information from expert opinion, by using the entropy principle and following a bayesian approach.  Computing these scores requires solving an exponential number of linear problems. In order to do so, we have developed an exact algorithm that reduces approximately half the time and the number of problems required. The method has been presented within a conventional framework wherein efficiency is assessed by means of the classical radial DEA models. Obviously, the approach proposed can be extended for use, for example, with external factors (non-controlled variables) whose incorporation into the analysis is to be considered, provided that an appropriate efficiency model is chosen in order to deal with such type of inputs/outputs (see \cite{Aristovnik2013} and \cite{Aristovnik2014} for some empirical work that considers external factors). In that sense, the extension of the method for use with other DEA models is also straightforward, because it only need from them the efficiency scores they provide. In particular, we can derive robust efficiency scores from non-radial DEA models (by using, for example, the enhanced Russell graph measure in \cite{ERG}), which would be of special interest when models that minimize the distance to the efficient frontier are used (see \cite{JPA2007}). In the latter case, the development of an heuristic algorithm for reducing computational time would become a key issue, due to the complexity of the efficiency models involved.

\section*{References}
\bibliographystyle{abbrvnat}

\end{document}